\documentclass[ a4paper
              , UKenglish
							, cleveref
							, autoref
							, thm-restate
							]{lipics-v2021}

\usepackage{fouche-sty/lipics-fouche}
\usepackage{eucal, pifont}
\title{Rosen's no-go theorem for regular categories}
\titlerunning{Rosen's no-go theorem}
\authorrunning{F. Loregian}

\newtheorem{perspective}[theorem]{Perspective}


	\Copyright{F. Loregian}

\author
  {Fosco Loregian}
	{Tallinn University of Technology, Estonia \and \url{http://tetrapharmakon.github.io}}
	{fosco.loregian@gmail.com}
	{https://orcid.org/0000-0003-3052-465X}
	{ESF funded Estonian IT Academy research measure (project 2014-2020.4.05.19-0001)}

\def\An{\cate{An}_A}
\def\rAn{\cate{rAn}_A}

\def\name#1{\lceil #1 \rceil}
\newcommand{\comma}[2]{(#1 \downarrow #2)}

\def\bsm{\left[\begin{smallmatrix}}
  \def\esm{\end{smallmatrix}\right]}

  \def\A{\bsm (0,b) &\to& \textcolor{red}{(1,b)} \\ && \\ \textcolor{red}{(0,a)} &\to& (1,a) \esm}
  \def\B{\bsm 0 &\to& 1 \\ &\searrow& \downarrow \\ && 2 \esm}
  \def\C{\{0\to 2\}}
  \def\D{\{0,   2\}}

\nolinenumbers
\hideLIPIcs
\def\Pos{\cate{Pos}}
\begin{document}
\maketitle
\begin{abstract}
  The famous biologist Robert Rosen argued for an intrinsic difference between biological and artificial life, supporting the claim that `living systems are not mechanisms'. This result --understood as the claim that life-like mechanisms are non-computable-- can be phrased as the non-existence of an equivalence between a category of `static'/analytic elements and a category of `variable'/synthetic elements.
  The property of a system of being synthetic, understood as being the glueing of `variable families' of \emph{analytica}, must imply that the latter class of objects does not retain sufficient information to describe said variability; we contribute to this thesis with an argument rooted in elementary category theory.
  Seen as such, Rosen's `proof' that no living system can be a mechanism arises from a tension between two contrapuntal needs: on one side, the necessity to consider (synthetically) variable families of systems; on the other, the necessity to describe a \emph{syntheticum} via a universally chosen \emph{analyticum}.
\end{abstract}
\maketitle
\section{Analytic and synthetic}
In the words of Rosen \cite{rosen1991life}, an analytic model is a set of \emph{observables} with a given \emph{resolution}. Such `resolution' is specified by a family of equivalence relations induced by functions $h : A \to X$ (each such function induces on $A$ an equivalence relation $\approx_h$ such that $a\approx_h a'$ if and only if $ha=ha'$); we call this `$h$-equivalence', or the equivalence relation `generated' by $h$.

Now, between the lines of \cite{rosen1991life} there is the definition of a \emph{category} of $A$-based analytica, where morphisms are determined by the refinement relation that naturally exists between equivalence relations generated by the various $h :A \to X$.

In the present document, we choose to slightly simplify this approach, considering the usual slice category of sets under $A$ instead, having objects the functions $h : A \to X$, and commutative triangles as its morphisms. In other words:
\begin{definition}[Analytic]\label{anali}
  Let $A$ be a set; the category $\An$ of \emph{functional $A$-based analytica} is the coslice category $A/\Set$.
\end{definition}
The typical morphism $h : \var{A}{X}\to \var{A}{Y}$ of $\An$ is a commutative triangle
\[\label{analy_are_fun}
  \begin{tikzcd}
    & A\ar[dr]\ar[dl]& \\
    X \ar[rr,"h"'] && Y
  \end{tikzcd}
\]
At first, it might seem that Rosen's category has way fewer morphisms than the one defined in \autoref{anali}; on the contrary, our definition is more restrictive, and Rosen's original category $\rAn$, that we choose to dub \emph{relational} analytica, can be recovered from the following observation.
\begin{remark}
  The category $\rAn$ of ($A$-based) relational analytica has objects the pairs
  \[\big( \var[k]{A}{Y},\firstblank\approx_k\firstblank\big),\]
  and there exists a morphism between $\big( \var[k]{A}{Y},\firstblank\approx_k\firstblank\big)$ and $\big( \var[h]{A}{Y},\firstblank\approx_h\firstblank\big)$ when $(\firstblank \approx_k\firstblank)\subseteq (\firstblank \approx_h\firstblank)$; under such assumption, there exists a well\hyp{}defined function between the quotient sets $A/_{\!\approx_k} \to A/_{\!\approx_h}$, sending a class of $\approx_k$-equivalence into its corresponding class of $\approx_h$-equivalence.

  Clearly, the existence of a function $f$ such that $f\circ k=h$ (so, the existence of a morphism in $\An$) ensures $\var[k]{A}{Y}\preceq \var[h]{A}{X}$, so the existence of a morphism in $\rAn$, but this is by no means necessary. Instead, the existence of an implication $(\firstblank \approx_k\firstblank)\subseteq (\firstblank \approx_h\firstblank)$ entails that there exists a \emph{relation} $h : X\pto Y$ closing diagram \eqref{analy_are_fun} (and this justifies the name `relational' analytica). We do not concentrate our attention of $\rAn$ for the moment, but we will expand on its structure at the end of the paper.
\end{remark}
On the other hand, considering fewer morphisms makes assessing global properties about the categories an easier task.

Rosen then proceeds to consider categories of \emph{synthetica}, aiming at modelling variability by indexing a syntheticum over some poset $P$; such a structure can be easily understood as a `comma construction' (cf. \cite[1.6.1]{Bor1}). If $A$ is a set, $\Pos$ is the category of partially ordered sets, and $1$ is the terminal category, we define:
\begin{definition}[Synthetic]\label{def_synth}
  The category $\cate{Syn}_A$ of $A$-synthetic models is the comma category in the diagram
  \[ \vcenter{\xymatrix{
    (\Pos\downarrow \An) \ar[r]\ar[d]_S & 1 \ar[d]^{\name{\An}}\\
    \Pos \ar[r]_J & \cate{Cat}\ultwocell<\omit>{}
    }} \]
  where $\name{\An}$ is the `name' of the category $\An$, i.e. the functor picking out $\An\in\Cat$.
\end{definition}
In simple terms, the category $\cate{Syn}_A$ is a category of ``families of analytic objects varying continuously(=functorially)'' with the elements of a poset $P$. Such families of living systems shall model inherently more complex systems than $\An$, and ultimately harbour ``life'', here understood as systems whose internal structure can't be examined piece-wise (one slice at a time, in the sense of \autoref{side_rem}) but instead \emph{globally}.
\begin{remark}
  Unraveling \autoref{def_synth}, we see that $(\Pos\downarrow \An)$ is the category whose
  \begin{itemize}
    \item objects are functors $f : P \to \An$, regarding a poset $P$ as a category;
    \item morphisms are monotone mappings $h : P \to Q$ such that the triangle of functors
          \[\vcenter{\xymatrix{
                P \ar[rr]^h \ar[dr]_f && Q \ar[dl]^g \\
                & A/\Set &
              }}\] is commutative.
  \end{itemize}
\end{remark}
\begin{remark}\label{side_rem}
  As a side remark, the functor $S$ has all sorts of excellent properties: it is a discrete fibration and each fiber consists of the category of functors $P \to A/\Set$; by the universal property of the comma construction, this category is equivalent to the category whose objects are pairs $(F:P\to \Set, \alpha : \Delta A \To F)$, where $\alpha$ is a cone for $F$ with domain $A$; again for the universal property of the limit of $F$, this amounts to a single function $A\to \lim F$: all in all, this means that we can present the fibres of $S$ as coslice categories on their own.
\end{remark}
\section{Synthetic is not analytic}
The results in \cite{rosen1991life,louie2007living} aim to convey the idea that `living systems are not mechanisms'; that is, any description of
these systems in terms of states would thus be incomplete. Rosen introduces his notion of a living
system in terms of general properties of relational models. Such relational
models defining a minimal organism are proved to be `inconsistent with the assumption that the corresponding system is fully defined by a state description' \cite{chu2006category}. `This also implies that a finite-state machine cannot implement the system. This insight substantiates Rosen's central result in his emph{Life itself}, namely, that there are systems that are not mechanisms, and that in fact nearly all systems fail to be mechanisms.'\footnote{\emph{ibid.}; even though we have taken their word to introduce the section, we must note that Chu and Ho's thesis has been harshly criticised by the researcher that above all carries Rosen's flame, see \cite{louie2007living}, and that our overall claim strays pretty far from \cite{chu2006category}'s thesis and methodology, actually arguing \emph{in favour} of Rosen's claim}

We now aim to phrase this `no-go' theorem in simple categorical terms as the non-existence of an equivalence of categories between analytica and synthetica, as follows:
\begin{theorem}[Rosen no-go theorem: analytic can't be synthetic]\label{nogo}
  There is no equivalence $\cate{Syn}_A\cong \An$.
\end{theorem}
We will deduce this result from a general theorem and a corollary, exploiting an exactness property of categories called \emph{regularity}. The idea behind this proof is that the category of analytica is properly contained into the category of synthetica, but in a wrong way that does not allow for the properties of $\An$ to be faithfully preserved.

\medskip
Recall the definition of regular category from \cite{Bor1}:
\begin{definition}[Regular category]
  A category is called \emph{regular} if it is finitely complete, and if the following two conditions are satisfied:
  \begin{itemize}
    \item considering the pullback of an arrow $f : X \to Y$ along itself,
          \[\begin{tikzcd}
              X\times_Y X \pb \ar[r, "p_0"]\ar[d, "p_1"'] & X\ar[d, "f"] \\
              X \ar[r, "f"']& Y
            \end{tikzcd}\]
          the pair $(p_0,p_1)$ has a coequaliser:
          \[
            \begin{tikzcd}
              X \times_Y X \ar[shift left=.5em, r, "p_0"] \ar[shift right=.5em, r, "p_1"']& X \ar[r] & Q
            \end{tikzcd}
          \]
    \item the class of regular epimorphisms (cf. \cite[4.3.1]{Bor1}) is stable under pullback; this means that if $u : A \to B$ is a regular epimorphism, and the square
          \[
            \begin{tikzcd}
              A'\ar[r]\ar[d, "u'"']\pb & A\ar[d, "u"] \\
              X \ar[r]& B
            \end{tikzcd}
          \]
          is a pullback, then $u'$ is still a regular epimorphism.
  \end{itemize}
\end{definition}
  The category of sets is easily seen to be regular; if a category is regular, so is the slice category of arrows $\clC/X$ and the coslice category $Y/\clC$ (\cite[A.5.5]{bourn}).

	Notorious examples of non-regular categories are: partially ordered sets, categories, and topological spaces. Quite niftily, a single counterexample fits all three cases, when interpreted as a poset, a category, or as finite topological space.
\begin{example}\label{alcune_no}
  let $A$ be the poset $\{a, b\} \times (0 \to 1)$; let $B$ be the poset $(0 \to 1 \to 2)$, and let $C$ be the poset $(0 \to 2)$. There is a regular epimorphism $p: A \to B$ obtained by identifying $(a, 1)$ with $(b, 0)$, and there is the evident inclusion $i: C \to B$. The pullback of $p$ along $i$ is the inclusion $\{0, 2\} \to (0 \to 2)$, which is certainly an epimorphism but not a regular epi. Hence regular epimorphisms in $Pos$ are not stable under pullback
  \[ \begin{tikzcd}
      \D \ar[r]\ar[d]\pb & \A\ar[d, two heads, "(\textcolor{red}{(0,a)}\equiv \textcolor{red}{(1,b)})"] \\
      \C \ar[r] & \B.
    \end{tikzcd}
  \]
  Interpreting posets as suitable categories, the same example works for $\Cat$, and also for preorders. On the other hand, the category of finite preorders is equivalent to the category of finite topological spaces, so this example shows that $\cate{Top}$ is not regular.
\end{example}
\begin{lemma}
  The comma category $\comma{J}{\clC}$ isn't regular. So, if $\clC$ is a regular category, the two can't be isomorphic.
\end{lemma}
\begin{proof}
  We show that regular epimorphisms are not pullback-stable because they are not pullback stable in $\Pos$. For this to work, a number of sanity checks are in order:
  \begin{enumtag}{c}
    \item \label{c_1} a regular epimorphism in $\Pos$ remains regular in $\comma{J}{\clC}$;
    \item \label{c_2} a pullback in $\comma{J}{\clC}$ is computed as in $\clC$;
  \end{enumtag}
  Once this has been verified, take a regular epimorphism $u : P \to Q$ in $\Pos$ whose pullback $u'$ along a map $q$ fails to be regular; now $u'$ is an epimorphism in $\comma{J}{\clC}$, but not regular, and it fits into a pullback diagram
  \[
    \begin{tikzcd}
      P' \ar[d, "u'"']\ar[r]\pb & P\ar[ddr, bend left]\ar[d, "u"] \\
      R \ar[r,"q"']\ar[drr, bend right]& Q\ar[dr] \\
      &&\clC
    \end{tikzcd}
  \]
  in $\comma{J}{\clC}$; thus, in $\comma{J}{\clC}$ the class of regular epimorphisms is not stable under pullback: so, $\comma{J}{\clC}$ is not regular.

  Verifying \ref{c_1} and \ref{c_2} is almost immediate: a regular epimorphism $u : P\to Q$ is an arrow that appears in a coequaliser
  \[ \vcenter{\xymatrix{
        A \ar@<.5em>[r]^f\ar@<-.5em>[r]_g & P \ar[r]^u & Q
      }} \]
  When $u,f,g$ are morphisms of $\comma{J}{\clC}$, the same diagram remains a coequaliser in $\comma{J}{\clC}$ when every vertex is endowed with a diagram to $\clC$. Now, a similar argument shows that a pullback in $\comma{J}{\clC}$ is computed as in $\clC$. (This can also be rephrased as a property of the functor $S$ of \autoref{def_synth}).
\end{proof}
From this, the claim in Theorem \ref{nogo} easily follows.
\section{There ain't no easy way out: so what?}

\begin{remark}
  The critical result \autoref{nogo} is phrased so that it is evident how no regular category can yield an equivalence between analytica and synthetica; the problem is the category of categories we choose to restrict our attention to: in this case, $\Pos$. Certainly, regularity is a reasonable tameness property for a category of analytica; but is it necessary to describe a living system?
\end{remark}
The situation requires an analytic (\dots pun not intended) approach: it's easy to isolate some of the various moving parts in the purported (non)equivalence between analytica and synthetica: consider the following drawing.
\begin{center}
  \begin{tikzpicture}
    \node (A/Set) {$A/\underset{\text{\ding{192}}}{\overset{\text{\ding{193}}}{\Set}}$};
    \node[right=.25cm of A/Set] (cong) {$\not\cong$};
    \node[right=1cm of A/Set] (comma) {$\underset{\text{\ding{194}}}{(J\downarrow A/\overset{\text{\ding{193}}}{\Set})}$};
    \node[right=1cm of comma] (fiber) {$\overset{\phantom{a}}{\underset{\text{\ding{195}}}{[P,A/{\Set}]}}$};
    \draw[left hook->] (fiber) -- (comma);
  \end{tikzpicture}
\end{center}
Such a representation of the situation makes it evident which parts of the equivalence we can tweak: for example, as already mentioned, our proof relies on a modified notion of $\An$, where objects are sets and \emph{functions} over $A$; what if \ding{192} we consider the bicategory $\cate{Rel}$ of relations, and the full subcategory of $A/\cate{Rel}$ given by the functional relations, as in diagram \eqref{analy_are_fun}? What if \ding{193} we replace $\Set$ with the category $\Cat$ of categories, and $A$ a (discrete or non-discrete) category? $\Cat$ is not regular (cf. \autoref{alcune_no}), nor it is the coslice $A/\Cat$. So, at least the obstruction to equivalence found in \autoref{nogo} vanishes. Alternatively, \ding{194} what if we consider a weaker universal object instead of the comma construction? Or finally \ding{195} what information can we get about the categories of analytica and synthetica from the knowledge of the fibres $\Cat(PA/\clC)$?

\medskip Let's examine each of these possibilities separately.
\begin{perspective}[Relational analytica]
  Sticking to Rosen's original idea, the category $\rAn$  of relational $A$-based analytica is defined as follows: it's the full subcategory of the (strong) coslice $A/\cate{Rel}$ on functional relations, and morphisms are relations $h : X\pto Y$ `compatible' with $A$:
  \[ \begin{tikzcd}
      & A \ar[dr, "g"]\ar[dl, "f"']& \\
      X \ar["{|}" marking, rr]&& Y
    \end{tikzcd}\]
  As interesting as it may seem, at the moment, this point of view raises more questions than it can solve; a thorough study of Rosen's theory from a relational/profunctorial point of view is the subject of current research \cite{mr_sys}, centred on a more categorically-informed account of $(M,R)$-systems. When moving from sets and functions to sets and relations, things become rather hairy because the ambient $\rAn$ is in no natural way a category but instead a bicategory; We shall then study its structure through 2-categorical types of machinery. (A blatant example of the fact that 1-dimensional niceness may fail to exist in $\rAn$ is that it lacks --even countable-- filtered colimits; this observation will play a central r\^ole in \cite{mr_sys}).
\end{perspective}
\begin{perspective}[Categorification]
  It is certainly worth exploring the following generalised setting: the category of $(2,A)$-analytica is the coslice $A/\Cat$, where $A$ is a (discrete or non-discrete) category. Now, since $\Cat$ is not regular, \autoref{nogo} does not apply; on the other hand, the two categories remain not equivalent because of the way colimits are computed in $\Pos$ and $\Cat$ (a coequaliser of posets performed in $\Cat$ yields a preorder, \emph{qua} thin category, whereas the coequaliser performed in $\Pos$ requires an additional step, a `posetal reflection' that quotients out by the equivalence relation $aRb\iff a\le b\le a$; this further quotient destroys the universal property of the $\Cat$-colimit).

  A natural alternative would be to consider \emph{adjunctions}, not equivalences; for example, since there is a natural embedding functor
  \[ i : \clC \to \comma{J}{\clC} \]
  sending an object $C\in\clC$ into its name $\name{C} : 1\to \clC$ (clearly an object of $\comma{J}{\clC}$), it would be interesting to find a `unity of opposites'
  \[
    \begin{tikzcd}
      A/\Cat \ar[r, hook, "i" description] & (J\downarrow A/\Cat).  \ar[l, shift right=.75em] \ar[l, shift left=.75em]
    \end{tikzcd}
  \]
  Unfortunately, $i$ does not preserve either products or coproducts so that it can be neither a right nor a left adjoint. We believe the question deserves attention, with an eye on the meaning of the (non)equivalences we are assessing.
\end{perspective}
\begin{perspective}[Weaker limits]
  Moving to a 2-categorical setting, it might be more natural to tweak the definition of analytic and synthetic system to form the pseudo-comma categories of commutative triangles, respectively:
  \[
  \vcenter{\xymatrix{
    &A\dtwocell<\omit>{\theta}\ar[dr]^g\ar[dl]_f& \\
    X \ar[rr]_h&& Y
  }}
  \] with $\theta : g \To hf$ an invertible natural transformation in $\Cat$; similarly, we can define the \emph{oplax} coslice category $A/\!\!/_{\!o}\Cat$ (where $\theta$ is not invertible) and the \emph{lax} coslice $A/\!\!/_{\!l}\Cat$, where instead of $\theta$ we consider transformations $\eta : hf\To g$. Such lax coslice category \cite[I.2.5]{Graya} appears to be the fiber over $A\in\Cat$ of a canonical 2--fibration $\text{src} : \Cat/\!\!/\Cat \to \Cat$ that proved to be useful in the classification of comprehension and quotient structures in categorical logic \cite{pammone_fib}
\end{perspective}
\begin{perspective}[Stick to a single fiber]
  We could study the particular case of (based or unbased) analytica where posets have good density and topological properties: this approach is tangential to the one conducted by Ehresmann and Vanbremeersch in their \cite{ehresmann1987hierarchical,ehresmann2009mens,ehresmann2019mes,ehresmann2007memory,bastiani1972categories}; in short, they consider a prestack $F : P \to \Cat$ where $P$ is a dense linear order (understood as a category of 'time instants', a natural choice for which is any interval $I\subset \bbR$ of the reals), and \emph{stacks} for the order topology. This induces nice comparison maps
  \[ F_t \to F_{t+\delta t} \to F_s \]
  for every $\delta t> 0$ and every $t \le s$ keeping track of the system's evolution globally.

  An important observation is that the fibres of the fibration $S$ of \eqref{def_synth}; $S$ resembles the \emph{fibration of presheaves}, defined as the Grothendieck fibration associated to the prestack $P\mapsto \Cat(P,\Set)$; the fibration of presheaves is a fibered topos, in the sense that every fibre is a topos and every reindexing functor is logical; albeit the category $\Cat(P,\An)$ is seldom a topos, we can still find that it is a fibered locally (finitely) presentable category: this means that
  \begin{itemize}
    \item Each fibre is a locally presentable category (the coslice of a locally presentable category is locally presentable, and by \autoref{side_rem}, the fibre is just the coslice $\Delta A / \Cat(P,\Set)$);
    \item each reindexing functor $f^* : S^{-1}Q \to S^{-1}P$ induced by $f : P \to Q$ is a right adjoint that preserves filtered colimits.
  \end{itemize}
\end{perspective}

% \section{Acknowledgements}
% The author was supported by the ESF funded Estonian IT Academy research measure (project 2014-2020.4.05.19-0001).

\bibliographystyle{plain}
\bibliography{allofthem}{}
% \putbib{allofthem}
\end{document}